\documentclass[10pt, a4paper]{article}
\usepackage{amssymb}
\usepackage{amsmath}
\usepackage{amsfonts}
\usepackage{amscd}
\usepackage{geometry}
\usepackage[all,cmtip]{xy}
\usepackage{makeidx}
\usepackage{titlesec}
\usepackage{fancyhdr}

\newcommand{\C}{\mathbb{C}}
\newcommand{\R}{\mathbb{R}}
\newcommand{\Q}{\mathbb{Q}}
\newcommand{\Z}{\mathbb{Z}}
\newcommand{\A}{\mathbb{A}}
\newcommand{\F}{\mathbb{F}}

\DeclareMathOperator{\HT}{HT}
\DeclareMathOperator{\Inert}{Inert}

\DeclareMathOperator{\Hom}{Hom}

\DeclareMathOperator{\Sym}{Sym}

\DeclareMathOperator{\Frob}{Frob}
\DeclareMathOperator{\Tr}{Tr}
\DeclareMathOperator{\Gal}{Gal}
\DeclareMathOperator{\GL}{GL}
\DeclareMathOperator{\GSp}{GSp}
\DeclareMathOperator{\PGL}{PGL}
\DeclareMathOperator{\SL}{SL}
\DeclareMathOperator{\SU}{SU}

\DeclareMathOperator{\PSL}{PSL}

\DeclareMathOperator{\SO}{SO}

\DeclareMathOperator{\GO}{GO}
\DeclareMathOperator{\Spin}{Spin}

\titleformat{\section}[hang]
{\normalfont\filright\large}{\thesection. }{0pt}
{\upshape\bfseries}

\titleformat{\subsection}[hang]
{\itshape}{\thesubsection \ - }{0pt}
{}

\usepackage{amsthm}
\theoremstyle{plain}
\newtheorem{theo}{Theorem}[section]
\newtheorem{prop}[theo]{Proposition}

\theoremstyle{remark}
\newtheorem{rema}[theo]{\sc Remark}
\theoremstyle{definition}

\title{On the images of certain $G_2$-valued automorphic Galois representations}
\author{\small ADRI\'AN ZENTENO \footnote{The author was partially supported by CONICYT Proyecto FONDECYT Postdoctorado No. 3190474}}

\date{\today}
\begin{document}
\maketitle

\begin{abstract}
In this paper we study the images of certain families $\{\rho_{\pi,\ell} \}_\ell$ of $G_2$-valued Galois representations of $\Gal(\overline{F}/F)$ associated to $L$-algebraic regular, self-dual, cuspidal automorphic representations $\pi$ of $\GL_7(\A_F)$, where $F$ is a totally real field. In particular, we prove that, under certain automorphic conditions, the images of the residual representations $\overline{\rho}_{\pi,\ell}$ are as large as possible for infinitely many primes $\ell$. Moreover, we apply our result to some examples constructed by Chenevier, Renard and Ta\"ibi.

\end{abstract}


\section{Introduction}

Let $F$ be a totally real field, $G_F:=\Gal(\overline{F}/F)$ be the absolute Galois group of $F$ and $\pi$ be an $L$-algebraic regular, self-dual, cuspidal automorphic representation of $\GL_n(\A_F)$.  Thanks to the work of Chenevier, Clozel, Harris, Kottwitz, Shin, Taylor and several others, we know that there exists a family $\{ \rho_{\pi,\ell} \}_\ell$ of continuous semi-simple Galois representations
\[
\rho_{\pi,\ell} : G_F \longrightarrow \GL_n(\overline{\Q}_{\pi,\ell})
\]
associated to $\pi$, such that Satake parameters an eigenvalues of Frobenius elements match. In particular, by the self-duality, the image of each $\rho_{\pi,\ell}$ is contained in $\GO_n(\overline{\Q}_\ell)$ or $\GSp_n(\overline{\Q}_\ell)$. 

A folklore conjecture, ensures that the images of the residual representations $\overline{\rho}_{\pi,\ell}$ should be as large as possible for almost all primes $\ell$ ($i.e.$ all but finitely many), unless there is an automorphic reason for it does not happen. In the 2-dimensional case, the conjecture was proven by Momose \cite{Mo81} and Ribet \cite{Ri85} when $\pi$ comes from a classical modular form and by Dimitrov \cite{Dim05} when $\pi$ comes from a Hilbert modular form. In this case, modular forms with complex multiplication (the automorphic reason) had to be excluded in order to obtain large image. When $\pi$ comes from a Hilbert-Siegel modular form of genus 2, the conjecture has been proved recently by Weiss \cite{Wei} \cite{Wei2}. In this case, CAP, endoscopic lifts, automorphic inductions and symmetric cube lifts need to be excluded to obtain large image.

In a recent work \cite{Ch19}, Chenevier has studied certain $L$-algebraic regular, self-dual, cuspidal automorphic representations $\pi = \otimes_v' \pi_v$ of $\GL_7(\A_F)$ of  weight $\{ -(h_\tau + k_\tau), -k_\tau, -h_\tau, 0, h_\tau,  k_\tau, h_\tau + k_\tau \}_{\tau \in \Hom_{\Q}(F, \C)}$, such that the 7-dimensional families of Galois representations  $\{\rho_{\pi,\ell} \}_\ell$ associated to them, are $G_2$-valued (Theorem \ref{Chen}). In this paper, we prove a weak version of the large image conjecture for these automorphic representations. 
More precisely, we prove that if the weight of $\pi$ is such that $k_\tau \neq 2h_\tau$ for some $\tau \in \Hom_{\Q}(F, \C)$, then there exists a positive Dirichlet density set of primes $\mathcal{L}$ such that for all $\ell \in \mathcal{L}$ the image of $\overline{\rho}_{\pi,\lambda}$ is isomorphic to $G_2(\F_{\ell^s})$ for some positive integer $s$ (Theorem \ref{Prin}). In fact, if we assume that for some finite place $v$, $\pi_v$ is square integrable, then the set of primes $\mathcal{L}$ has Dirichlet density 1 (Theorem \ref{vfue}).

We remark that the condition imposed on the weight of $\pi$ is in order to exclude sixth symmetric power lifts. However, as we will see in Section \ref{sec3}, there are cuspidal automorphic representations of weight $\{ -3h_\tau, -2h_\tau, -h_\tau, 0, h_\tau,  2h_\tau, 3h_\tau \}_{\tau \in \Hom_{\Q}(F, \C)}$ such that they are not sixth symmetric power lifts.
When $F=\Q$, by using Serre's modularity conjecture, we prove that if $\pi$ is not a sixth symmetric power lift, then there exists a positive Dirichlet density set of primes $\mathcal{L}$ (in fact of Dirichlet density one, if $\pi_p$ is square integrable for some prime $p$) such that for all $\ell \in \mathcal{L}$ the image of $\overline{\rho}_{\pi,\ell}$ is isomorphic to $G_2(\F_{\ell^s})$ for some positive integer $s$ (Theorem \ref{Prin2}). Finally, we show that the 11 examples of cuspidal automorphic representations of $\GL_7(\A_\Q)$ of level one given in \cite[$\S$ 6.11]{Ch19}, satisfy some of our results (Proposition \ref{ejemp} and Remark \ref{ejem2}).

The proof of our results follows the line of \cite{Di02} and \cite{DZ20}, in the sense that our main tools are: some recent results about residual irreducibility of compatible systems of Galois representations \cite{BLGGT14} \cite{PT15}, the classification of the maximal subgroups of $G_2(\F_{\ell^r})$ \cite{Kle88}, and Fontaine-Laffaille theory \cite{FL82} \cite{Bar20}.

To the best of our knowledge, the only $G_2$-valued automorphic Galois representations that have been studied in this direction have been those associated to some examples of cuspidal automorphic representation of $\GL_7(\A_\Q)$ with certain prescribed local ramification. See \cite{KLS10} and \cite{MS14}.



\section{Preliminaries on 7-dimensional Galois representations}\label{sec1}

In this section we review some definitions and results about Galois representations associated to cuspidal automorphic representations of $\GL_7$ over totally real fields.  Our main references are \cite{BG14} and \cite{Bar20}.

Let $F$ be a totally real field and $\pi = \otimes_v ' \pi_v$ be a cuspidal automorphic representation of $\GL_7(\A_F)$. 
Let $v$ be an Archimedean place of $F$ and $\tau :F \hookrightarrow \C$ be the embedding inducing $v$. 
Langlands classification associates to $\pi_\tau = \pi_v$ a semi-simple representation $\phi_\tau: W_\R \rightarrow \GL_7(\C)$ of the Weil group $W_\R$.
We will say that $\pi_\tau$ is \emph{L-algebraic}, if the restriction of $\phi_\tau$ to the Weil group $W_\C = \C^\times$ is of the form 
\[
\phi_\tau \vert_{\C^\times} = \chi_{\tau,1} \oplus \cdots \oplus \chi_{\tau,7}
\]
where  $\chi_{\tau,i}: \C^\times \rightarrow \C^\times$ are characters such that 
\[
\chi_{\tau, i}(z) = z^{a_{\tau,i}} \overline{z}^{b_{\tau,i}}
\] 
with $a_{\tau,i}, b_{\tau,i} \in \Z$. Let $\Z^7_+$ be the set of $7$-tuples of integers $\{ \alpha_1, \ldots, \alpha_7 \} \in \Z^7$ such that $\alpha_1 \leq \ldots \leq  \alpha_7$. After reordering indices, we will refer to the 7-tuple 
$\{ a_{\tau,1}, \cdots a_{\tau,7} \} \in \Z^7_+$ 
as the weight of $\pi_\tau$. Moreover, we will say that $\pi_\tau$ is \emph{regular} if the $a_{\tau,i}$ are distinct. Finally, we will say that $\pi$ is \emph{L-algebraic regular} of weight 
$\{a_{\tau,1} , \ldots , a_{\tau,7} \}_\tau \in (\Z^7_+)^{\Hom_\Q (F,\C)}$,
if for each $\tau \in \Hom_\Q (F, \C)$, $\pi_\tau$ is $L$-algebraic and regular of weight $\{ a_{\tau,1}, \ldots , a_{\tau,7} \} \in \Z^7_+$.

Let $\pi = \otimes' _v \pi_v$ be an $L$-algebraic regular cuspidal automorphic representation of $\GL_7(\A_F)$ and $S_\pi$ be the finite set of finite places $v$ of $F$ at which $\pi_v$ is ramified. From now on, for each prime $\ell$, we fix an isomorphism $\iota: \C \xrightarrow{\sim} \overline{\Q}_\ell$.
If we assume that $\pi$ is \emph{self-dual} ($i.e.$ $\pi^\vee \simeq \pi$) it can be proved that, for each prime $\ell$, there exists a continuous semi-simple representation
\[
\rho_{\pi,\ell} : G_F \longrightarrow \GL_7(\overline{\Q}_{\ell})
\]
such that if $v \notin S_\pi$ and $v \nmid \ell$ then $\rho_{\pi,\ell}$ is unramified at $v$ and the characteristic polynomial of a Frobenius element $\Frob_v$ satisfies
\[
\det(X - \rho _{\pi,\ell} (\Frob_v)) = \iota \det (X-c(\pi_v)),
\]
where $c(\pi_v)$ is the Satake parameter of $\pi_v$ viewed as a semi-simple conjugacy class in $\GL_7(\C)$, while if $v \vert \ell$ then $\rho_{\pi,\ell} \vert_{G_{F_v}}$ is de Rham and in fact crystalline when $v \notin S_\pi$ \cite[Theorem 3.1.2]{Ta16}. 

Now, we will explain the relationship between the Hodge-Tate numbers of $\rho_{\pi,\ell} \vert_{G_{F_v}}$  and the inertial weights of its reduction modulo $\ell$. As $ \rho_{\pi,\ell} \vert_{G_{F_v}}$ is de Rham (then, by definition, Hodge-Tate) at $v \vert \ell$, for each embedding $\tau: F_v \hookrightarrow \overline{\Q}_\ell$, we can attach to $ \rho_{\pi,\ell} \vert_{G_{F_v}}$ a multiset of integers 
$\HT_{\tau} ( \rho_{\pi,\ell} \vert_{G_{F_v}}) = \{\alpha_{\tau,1} , \ldots, \alpha_{\tau,7} \} \in \Z^7_+$
called the $\tau$-Hodge-Tate numbers of $ \rho_{\pi,\ell} \vert_{G_{F_v}}$. These numbers can be obtained from the weight of $\pi$ as follows. Let $\{a_{\tau,1} , \ldots , a_{\tau,7} \}_\tau \in (\Z^7_+)^{\Hom_\Q (F,\C)}$ be the weight of $\pi$. Identifying $\{ (v,\tau): v \vert \ell, \; \tau \in \Hom_{\Q_\ell} (F_v, \overline{\Q}_\ell) \}$ with $\Hom_{\Q}(F, \C)$ via the fixed isomorphism $\iota$, we have that
\[
\HT_{\tau} ( \rho_{\pi,\ell} \vert_{G_{F_v}}) = \{a_{\tau,1} , \ldots, a_{\tau,7} \} \in \Z^7_+,
\]
where $\tau$ in the right side is the embedding associate to the pair $(v,\tau)$.

On the other hand, let $\overline{\rho}_{\pi,\ell} : G_{F} \rightarrow \GL_7(\overline{\F}_\ell)$ be the mod $\ell$ reduction of $\rho_{\pi,\ell}$ and $\F_v$ be the residue field of $F_v$.  If we assume that $F$ is unramified at $\ell$, we can attach to $\overline{\rho}_{\pi,\ell} \vert_{G_{F_v}}$ a subset 
\[
\Inert (\overline{\rho}_{\pi,\ell} \vert_{G_{F_v}}) \subset (\Z^7_+)^{\Hom_{\F_\ell}(\F_v,\overline{\F}_\ell)}
\]
called the set of inertial weights of $\overline{\rho}_{\pi,\ell} \vert_{G_{F_v}}$. These weights, which are an analogue of Hodge-Tate numbers for mod $\ell$ representations, only depend on the restriction to the inertia subgroup $I_{F_v} \subset G_{F_v}$  of the semi-simplification of $\overline{\rho}_{\pi,\ell} \vert_{G_{F_v}}$. Let
\[
\HT ( \rho_{\pi,\ell} \vert_{G_{F_v}}) = \{a_{\tau,1} , \ldots, a_{\tau,7} \}_\tau \in (\Z^7_+)^{\Hom_{\Q_\ell}(F_v,\overline{\Q}_\ell)}
\]
be the Hodge-Tate numbers of $\rho_{\pi,\ell} \vert_{G_{F_v}}$. Note that, as we are assuming that $F$ is unramified at $\ell$, we can index the Hodge-Tate numbers of $\rho_{\pi,\ell} \vert_{G_{F_v}}$ by embeddings $\F_v \hookrightarrow \overline{\Q}_\ell$ rather than embeddings $F_v \hookrightarrow \overline{\Q}_\ell$. Thus, we can see $\HT (\rho_{\pi,\ell} \vert_{G_{F_v}})$ as an element of $(\Z^7_+)^{\Hom_{\F_\ell}(\F_v,\overline{\F}_\ell)}$. By Fontaine-Lafaille theory \cite[Theorem 1.0.1]{Bar20} we have that, if $\rho_{\pi,\ell} \vert_{G_{F_v}}$ is crystalline and $a_{\tau,n} - a_{\tau,1} \leq p$ for all $\tau \in \Hom_{\Q_\ell} (F_v, \overline{\Q}_\ell)$, then 
\[
\HT ( \rho_{\pi,\ell} \vert_{G_{F_v}}) \in \Inert (\overline{\rho}_{\pi,\ell} \vert_{G_{F_v}}).
\]


\section{$G_2$-valued Galois representations with large image}\label{sec2}

Let $G_2$ be the automorphism group scheme of the standard split octonion algebra over $\Z$.  It is well known that, for any algebraically closed field $k$ of characteristic 0, there is a unique (up to isomorphism) irreducible $k$-linear algebraic representation $\sigma : G_2(k) \rightarrow \GL_7(k)$.
Using the result of the previous Section on the existence of Galois representations associated to self-dual cuspidal automorphic representations of $\GL_7(\A_F)$, Chenevier \cite[Corollary 6.5, Corollary 6.10]{Ch19} proved the following result:

\begin{theo}\label{Chen}
Let $\pi = \otimes'_v \pi_v$ be an $L$-algebraic regular, self-dual, cuspidal automorphic representation of $\GL_7(\A_F)$ and
assume that, for almost all finite places $v \notin S_\pi$, the Satake parameter $c(\pi_v)$ of $\pi_v$ is the conjugacy class of an element in $\sigma(G_2(\C))$.
Then, for each prime $\ell$, there exists a continuous semi-simple representation
\[
\rho_{\pi,\ell} : G_{F} \longrightarrow G_2(\overline{\Q}_{\ell})
\]
such that
\begin{itemize}
\item if $v \notin S_\pi$ and $v \nmid \ell$ then $\rho_{\pi,\ell}$ is unramified and
\[
\det(X-\sigma(\rho_{\pi,\ell} (\Frob_v))) = \iota \det (X-c(\pi_v)),
\]
\item while if $v \vert \ell$ then $\rho_{\pi,\ell} \vert_{G_{F_v}}$ is de Rham and in fact crystalline when $v \notin S_\pi$. 
\end{itemize}
Moreover, the weight of $\pi$ is of the form 
\[
\{ -(h_\tau + k_\tau), -k_\tau, -h_\tau, 0, h_\tau,  k_\tau, h_\tau + k_\tau \}_\tau \in (\Z^7_+)^{ \Hom_{\Q}(F, \C)}.
\]
\end{theo}

Let $\overline{\rho}_{\pi,\ell} : G_{F} \rightarrow G_2(\overline{\F}_\ell)$ be the semi-simplification of the mod $\ell$ reduction of $\rho_{\pi,\ell}$. This representation is usually called the \emph{residual representation} of $\rho_{\pi,\ell}$. The main goal of this paper is to prove the following result.

\begin{theo}\label{Prin}
Let $\pi$ be an $L$-algebraic regular, self-dual, cuspidal automorphic representation of $\GL_7(\A_F)$ as in Theorem \ref{Chen} and assume that the weight of $\pi$ is such that $k_\tau \neq 2h_\tau$ for some $\tau \in \Hom_{\Q}(F, \C)$. Then, there exists a positive Dirichlet density set of primes $\mathcal{L}$ such that for all $\ell \in \mathcal{L}$ the image of $\overline{\rho}_{\pi,\ell}$ is isomorphic to $G_2(\F_{\ell^s})$ for some positive integer $s$.
\end{theo}

The proof of this theorem follows the structure of \cite{Di02} and \cite{DZ20}. Then, as in loc. cit., the proof of Theorem \ref{Prin} is done by considering the possible images of $\overline{\rho}_{\pi,\ell}$ given by the maximal subgroups of $G_2(\F_{\ell^r})$. Such subgroups were classified by Kleidman in \cite{Kle88}.

\begin{prop}\label{clas}
Let $\F_q$ be a finite field of characteristic $\ell>11$ and $q=\ell^r$. Then, the maximal proper subgroups of $G_2(\F_q)$ are as follows:
\begin{enumerate}
\item maximal parabolic subgroups;
\item $\SL_3(\F_q){:}2$ and $\SU_3(\F_q) {:} 2$;
\item $(\SL_2(\F_q) \circ \SL_2(\F_q)).2$;
\item $\PGL_2(\F_q)$, $\ell \geq 7$, $q\geq 11$;
\item $2^{3.}\PSL_3(\F_2)$, $\PSL_2(\F_{13})$, $\PSL_2(\F_8)$, $G_2(\F_2)$;
\item $G_2(\F_{q_0})$, $q = q_0^s$, $s$ prime.
\end{enumerate}
\end{prop}

The assumption in the characteristic of $\F_q$ is in order to avoid the difficulties associated to the extra maximal subgroups appearing in characteristic 2, 3  and 11 (see for example \cite[Section 4.3]{Wil09}). We remark that this restriction does not matter for our purposes because we only need to know the previous classification for a positive Dirichlet density set of primes.

Another tool that is widely used in \cite{Di02} and \cite{DZ20}, is Fontaine–Laffaille theory. More precisely, we will use the following result which follows from the previous section and Theorem \ref{Chen}.

\begin{prop}\label{Iner}
Let $\pi = \otimes '_v \pi_v$ be an $L$-algebraic regular, self-dual, cuspidal automorphic representation of $\GL_7(\A_\Q)$ as in Theorem \ref{Chen}. Then, for each finite place $v \vert \ell$ and each $\tau \in \Hom_{\Q_\ell} (F_v, \overline{\Q}_\ell)$ we have that
\[
\HT_{\tau} (\rho_{\pi, \ell} \vert _{G_{F_v}}) = \{ -(h_\tau + k_\tau), -k_\tau, -h_\tau, 0, h_\tau,  k_\tau, h_\tau + k_\tau \} \in \Z^7_+
\]
Moreover, if $v \notin S_\pi$ and $2h_\tau + 2k_\tau \leq \ell$ for all $\tau \in \Hom_{\Q_\ell} (F_v, \overline{\Q}_\ell)$, then
\[
\HT(\rho_{\pi, \ell} \vert _{G_{F_v}}) \in \Inert(\overline{\rho}_{\pi, \ell} \vert _{G_{F_v}}).
\]
\end{prop}

Now, we are ready to prove Theorem \ref{Prin}. Our proof will be given by showing that the image of $\overline{\rho}_{\pi,\ell}$ is not contained in any subgroup lying in cases $i)-v)$ of Proposition \ref{clas}.

\paragraph*{Proof of Theorem \ref{Prin}.}
Let $\pi$ be an $L$-algebraic regular, self-dual, cuspidal automorphic representation of $\GL_7(\A_F)$ as in Theorem \ref{Chen}. By Theorem 1.7 of \cite{PT15}, we have that there exists a positive  Dirichlet density set of primes $\mathcal{L}''$ such that for all $\ell \in \mathcal{L}''$ the representation $\rho_{\pi,\ell}$ is irreducible. 
Then, it can be proved, by an identical argument  to the proof of Proposition 5.3.2 of \cite{BLGGT14}, that there is a positive Dirichlet density set of primes $\mathcal{L}' \subset \mathcal{L}''$ (obtained by removing a finite number of primes from $\mathcal{L}''$) such that $\overline{\rho}_{\pi,\ell}$ is irreducible for all $\ell \in \mathcal{L}'$. 
Then, if $\ell \in \mathcal{L}'$, the image of $\overline{\rho}_{\pi,\ell}$ cannot be contained in a maximal subgroup in cases $i)-iii)$ of Proposition \ref{clas} because they are reducible groups.

Now, we will deal with case $iv)$ of Proposition \ref{clas}. In this case, $\PGL_2(\F_q)$ fits into $G_2(\F_q)$ via $\Sym^6: \PGL_2 \rightarrow G_2$. Then, if $G_\ell := \mbox{Im}(\overline{\rho}_{\pi,\ell})$ is contained in $\Sym^6 (\PGL_2(\F_q))$, the elements of $G_\ell$ are of the form
\[
\Sym^6 
\begin{pmatrix}
x & * \\
* & y 
\end{pmatrix} = \begin{pmatrix}
x^6 & * & * & * & * & * & * \\
* & x^5 y & * & * & * & * & * \\
* & * & x^4 y^2 & * & * & * & * \\
* & * & * & x^3 y^3 & * & * & * \\
* & * & * & * & x^2 y^4 & * & * \\
* & * & * & * & * & x y^5 & * \\
* & * & * & * & * & * & y^6 
\end{pmatrix}
\]
where $x,y \in \overline{\F}_\ell$. Then, we can deduce that
\[
(x^{(6-m)} y^m)(x^{(6-m)-2} y^{m+2}) = (x^{(6-m)-1} y^{m+1})^2
\]
for $0 \leq m \leq 4$. From these equalities, we have that for all $\ell$ sufficiently large and any $v \vert \ell$, the inertial weights $\{ \alpha_{\tau,1}, \ldots, \alpha_{\tau,7} \}_\tau \in \Inert(\overline{\rho}_{\pi, \ell} \vert _{G_{F_v}})$ should satisfy the following relation
\begin{equation}\label{pesos}
\alpha_{\tau,i}+\alpha_{\tau, i+2} = \alpha_{\tau, i+1}
\end{equation}
for $1 \leq i \leq 5$. In particular, if $\ell$ is such that $v \notin S_\pi$ and $2h_\tau + 2k_\tau \leq \ell$ for each $\tau \in \Hom_{\Q_\ell} (F_v, \overline{\Q}_\ell)$, we have by Proposition \ref{Iner} that the $\tau$-Hodge-Tate numbers $\HT_\tau (\rho_{\pi, \ell} \vert _{G_{F_v}}) =  \{ -(h_\tau + k_\tau), -k_\tau, -h_\tau, 0, h_\tau,  k_\tau, h_\tau + k_\tau \} = \{ \alpha_{\tau,1} , \ldots, \alpha_{\tau,7} \}$ should satisfy (\ref{pesos}). However, that only happens if $k_\tau = 2h_\tau$ for all $v \vert \ell$ and any $\tau \in \Hom_{\Q_\ell} (F_v, \overline{\Q}_\ell)$.
Then, by our assumption on the weight of $\pi$, we have that the image of $\overline{\rho}_{\pi, \ell}$ cannot be contained in $\Sym^6 (\PGL_2(\F_q))$ when $\ell$ is sufficiently large.

Finally, if the image of $\overline{\rho}_{\pi,\ell}$ is contained in one of the maximal subgroups in case $v)$ of Proposition \ref{clas}, the order of $G_\ell$ is bounded independently of $\ell$. Then, by \cite[Lemma 5.3]{CG13}, if $\ell$ is large enough, the image of $\overline{\rho}_{\pi,\ell}$ cannot be contained in one of these maximal subgroups.

Therefore, there is a positive Dirichlet density set of primes $\mathcal{L}$ (obtained after removing possibly a finite number of small primes from $\mathcal{L'}$) such that for all $\ell \in \mathcal{L}$ the image of $\overline{\rho}_{\pi,\ell}$ is isomorphic to $G_2(\F_{\ell^s})$ for some positive integer $s$.

\qed

\bigskip

We remark that if we allow certain local ramification behavior in our automorphic representations, we can obtain a strong version of Theorem \ref{Prin}.

\begin{theo}\label{vfue}
Let $\pi = \otimes' _v \pi_v$ be an $L$-algebraic regular, self-dual, cuspidal automorphic representation of $\GL_7(\A_F)$ as in Theorem \ref{Prin} and assume that for some finite place $v$, $\pi_v$ is square integrable. Then there exists a set of primes $\mathcal{L}$ of Dirichlet density 1 such that for all  $\ell \in \mathcal{L}$ the image of $\overline{\rho}_{\pi,\ell}$ is isomorphic to $G_2(\F_{\ell^s})$ for some positive integer $s$.
\end{theo}

\begin{proof}
Let $\ell$ be a rational prime such that $v \nmid \ell$.
As we are assuming that $\pi_v$ is square integrable, from Corollary B of \cite{TY07}, we have that $\rho_{\pi,\ell}$ is irreducible. 
By Proposition 5.3.2 of \cite{BLGGT14}, there exists a set of primes $\mathcal{L}'$ of Dirichlet density 1, such that for all $\ell \in \mathcal{L}'$, $\overline{\rho}_{\pi,\ell}$ is irreducible. The rest of the proof is exactly the same as the proof of Theorem \ref{Prin}.
In particular, the set of primes $\mathcal{L}$ of Dirichlet density 1 is obtained by removing a finite number of primes from $\mathcal{L'}$ as in the proof of Theorem \ref{Prin}.

\end{proof}

Finally, we remark that Magaard and Savin \cite{MS14} have used (before the appearance of Chenevier's work) this kind local behavior to construct a self-dual cuspidal automorphic representation $\pi$ of $\GL_7(\A_\Q)$ (unramified outside $5$ and such that $\pi_5$ is Steinberg), such that the image of the residual representations $\overline{\rho}_{\pi,\ell}:G_\Q \rightarrow \GL_7(\overline{\F}_\ell)$ associated to $\pi$ are equal to $G_2(\F_\ell)$ for an explicit set of primes of Dirichlet density at least $1/18$.


\section{Some examples and improvements in the case $F=\Q$}\label{sec3}

When $F=\Q$, examples of cuspidal automorphic representations satisfying the assumptions of Theorem \ref{Prin} can be obtained from the computations of Chenevier, Renard \cite{CR15} and Ta\"ibi \cite{Ta17}.

More precisely, let $h,k,t \in \Z$, with $0<h<k<t$, and $O_{o}(t,k,h)$ be the set of $L$-algebraic regular, self-dual, cuspidal automorphic representations of $\GL_7(\A_\Q)$ of level one (i.e. which are everywhere unramified) and weight $\{-t,-k,-h,0,h,k,t \} \in \Z^7_+$. It follows from Theorem 1 of \cite{HC68} that the cardinality of $O_{o}(t,k,h)$ is finite. In the extended version of Table 22 of \cite{Ta17}, Ta\"ibi  compute explicitly the cardinality of $O_{o}(t,k,h)$ for all $0<h<k<t \leq 22$. Then, from Ta\"ibi's computations and Theorem 6.12 of \cite{Ch19}, we have the following result. 

\begin{prop}\label{ejemp}
Let $\mathcal{G}_{2}(h,k)$ be the subset of $L$-algebraic regular, self-dual, cuspidal automorphic representations of $O_{o}(h+k,k,h)$ satisfying the assumptions of Theorem \ref{Prin}. If 
\[
(h,k) \in \{ (5,8), (3,10), (5,9),(4,10),(2,12),(7,8),(4,11),(1,14), (1,16), (1,17) \},
\]
then $\vert \mathcal{G}_{2}(h,k) \vert = 1$.
\end{prop}

On the other hand, let $\varpi$ be an $L$-algebraic regular, cuspidal automorphic representation of $\GL_2(\A_\Q)$, which in fact corresponds to a twist of a cuspidal Hecke eigenform of weight at least 2. By Langlands Functoriality, the sixth symmetric power lifting $\Sym^6(\varpi)$  is an $L$-algebraic regular, self-dual, cuspidal automorphic representation of $\GL_7(\A_\Q)$ \cite[Theorem 6.1]{CT17}. Then, we will say that an $L$-algebraic regular, self-dual, cuspidal automorphic representation $\pi$ of $\GL_7(\A_\Q)$ is a \emph{sixth symmetric power lift} if there is an $L$-algebraic regular, cuspidal automorphic representation $\varpi$ of $\GL_2(\A_\Q)$  such that, for any prime $\ell$,
\[
\rho_{\pi,\ell} \cong \Sym^6 (\sigma_{\varpi,\ell}),
\]
where $\sigma_{\varpi,\ell}: G_\Q \rightarrow \GL_2(\overline{\Q}_\ell)$ is the $\ell$-adic Galois representation associated to $\varpi$. We remark that if $\pi$ is a sixth symmetric power lift, the weight of $\pi$ must be of the form $\{ -3h,-2h,-h, 0, h, 2h, 3h \}$. 
Thus, the automorphic representations considered  in Theorem \ref{Prin} cannot be sixth symmetric power lifts. Thanks to Serre's modularity conjecture, which is a theorem when $F = \Q$ (see \cite{KW09a}, \cite{KW09b} and \cite{Di12}), we have the following result.

\begin{theo}\label{Prin2}
Let $\pi$ be an $L$-algebraic regular, self-dual, cuspidal automorphic representation of $\GL_7(\A_\Q)$ as in Theorem \ref{Chen}. If $\pi$ is not a sixth symmetric power lift, then there exists a positive Dirichlet density set of primes $\mathcal{L}$ such that for all $\ell \in \mathcal{L}$ the image of $\overline{\rho}_{\pi,\lambda}$ is isomorphic to $G_2(\F_{\ell^s})$ for some positive integer $s$.
Moreover, if $\pi_p$ is square integrable for some prime $p$, then $\mathcal{L}$ has Dirichlet density 1.
\end{theo}

\begin{proof}
As in Theorem \ref{Prin}, the proof is given by showing that the image of $\overline{\rho}_{\pi,\ell}$ cannot be contained in any subgroup lying in cases $i)-v)$ of Proposition \ref{clas}. 

Let $\{-(h+k), -k, -h, 0,h,k, h+k \}$ be the weight of $\pi$ and assume that $k=2h$. The case $k \neq 2h$ was dealt in Theorem \ref{Prin}.
Note that cases $i)-iii)$ can be dealt in exactly the same way as in the proof of Theorem \ref{Prin}. Then, there is a positive Dirichlet density set of primes $\mathcal{L'}$ such that, for all $\ell \in \mathcal{L}'$, $\overline{\rho}_{\pi,\ell}$ is irreducible. 

Now, let $\ell \in \mathcal{L}'$ and assume that the image of $\overline{\rho}_{\pi,\ell}$ is contained in a maximal subgroup lying in case $iv)$ of Proposition \ref{clas}. Then
\[
\overline{\rho}_{\pi, \ell} \simeq \Sym^6(\overline{\sigma}_\ell),
\]
where $\overline{\sigma}_\ell: G_\Q \rightarrow \GL_2(\overline{\F}_\ell)$ is a two-dimensional irreducible Galois representation. From \cite[Proposition 1]{Tay12}(see also \cite{Ta16} and \cite{CH16}), we have that, if  $c \in G_F$ is a complex conjugation, then $ \Tr (\rho_{\pi, \ell}(c)) = \pm 1$. Thus, by the structure of $\Sym^6$, we have that $\overline{\sigma}_\lambda$ is odd. 
Moreover, from the explicit description of $\Sym^6$, the weight of $\pi$ and Proposition \ref{Iner}, we have that, if  $\ell \notin S_\pi$ and $6h \leq \ell$, $\overline{\sigma}_\ell$ has an inertial weight of the form $\{ -\frac{h}{2}, \frac{h}{2} \}$. 
Hence, by Serre's modularity conjecture, there is a cuspidal Hecke eigenform $f$, of weight $h+1\geq 2$ and level bounded independently of $\ell$, such that
\begin{equation}\label{simi}
\rho_{\pi, \ell} \equiv \Sym^6(\sigma_{\varpi_{f},\lambda}) \mod \ell,
\end{equation}
where $\varpi_f$ is the $L$-algebraic regular, cuspidal automorphic representation of $\GL_2(\A_\Q)$ corresponding to $f$. We remark that the set of such cuspidal Hecke eigenform (with a fixed weight and bounded level) is finite.
Then, if the congruence (\ref{simi}) is satisfied for infinitely many primes $\ell$, by Dirichlet principle, we have that there exist a fixed cuspidal Hecke eigenform $f$ such that
\[
\rho_{\pi, \lambda} \equiv \Sym^6(\sigma_{\varpi_f,\lambda}) \mod \ell,
\]
for infinitely many primes $\ell$. Therefore, by Chevotarev's density theorem, it follows that 
\[
\rho_{\pi, \lambda} \simeq \Sym^6(\sigma_{\varpi_f,\lambda}),
\]
for all primes $\ell$. Thus, $\pi$ is a sixth symmetric power lift, contradicting our assumption on $\pi$.

Finally, case $v)$ of Proposition \ref{clas} can be dealt as in the proof of Theorem \ref{Prin} and the set of primes $\mathcal{L} \subset \mathcal{L}'$ of positive Dirichlet density can be obtained by removing at most a finite number of small primes from $\mathcal{L'}$. Moreover, if $\pi_p$ is square integrable for some prime $p$, we can proceed exactly as in Theorem \ref{vfue}.

\end{proof}

\begin{rema}\label{ejem2}
From the computations of Chenevier, Renard and Ta\"ibi, we can obtain an example of cuspidal automorphic representation satisfying Theorem \ref{Prin2} but not Theorem \ref{Prin}. More precisely, from Theorem 6.12 of \cite{Ch19} we have that there exists an $L$-algebraic cuspidal, self-dual, cuspidal automorphic representation of $\GL_7(\A_\Q)$ satisfying Theorem \ref{Chen}, but of weight $(4,8)$, then it does not satisfy Theorem \ref{Prin}. 
However, this cuspidal automorphic representation satisfies Theorem \ref{Prin2} because it is not a sixth symmetric power lift. If it were a sixth symmetric power lift, by the discussion in the proof of Theorem \ref{Prin2}, it should come from a cuspidal Hecke eigenform $f$ of weight $5$ and level 1, which does not exist. 
\end{rema}

\paragraph*{Acknowledgments:} I would like to thank Ariel Weiss for very useful comments and suggestions on a previous version of this paper. I also thank Luis Lomelí for useful discussions about the exceptional group $G_2$.



\textsc{.\\
Instituto de Matem\'aticas \\ Pontificia Universidad Cat\' olica de Valpara\'iso \\ Blanco Viel 596, Cerro Bar\'on \\ Valpara\'iso, Chile} \\
E-mail: \texttt{adrian.zenteno@pucv.cl}

\end{document}